\documentclass[12pt]{amsart}

\usepackage{graphicx}
\usepackage{fancyhdr}
\usepackage{amsthm}
\usepackage{lipsum}
\usepackage{amsfonts, amsmath}
\usepackage[colorlinks=true,linkcolor=red,citecolor=blue]{hyperref}
\usepackage[margin=1.25in]{geometry}
\usepackage{tikz}
\usepackage{setspace}
\usepackage[square,numbers]{natbib}

\usepackage{float}
\usepackage[normalem]{ulem}
\usepackage{verbatim}

\newtheorem{thm}{Theorem}[section]

\newtheorem{lem}[thm]{Lemma}
\newtheorem{prop}[thm]{Proposition}

\newtheorem{quest}[thm]{Question}

\newtheorem{innerthmrep}{Theorem}
\newenvironment{thmrep}[1]
  {\renewcommand\theinnerthmrep{#1}\innerthmrep}
  {\endinnerthmrep}

\newtheorem{innerprprep}{Proposition}
\newenvironment{prprep}[1]
  {\renewcommand\theinnerprprep{#1}\innerprprep}
  {\endinnerprprep}
  
\newtheorem{defn}[thm]{Definition}

\newtheorem{remk}[thm]{Remark}

\newcommand{\Zbb}{ {\mathbb Z}}

\newcommand{\Ztwo}{\mathbb{Z}/2\mathbb{Z}}
\newcommand{\Bpartial}{\overline{\partial}}

\newcommand{\Mcal}{\mathcal{M}}
\newcommand{\Mcalg}{\mathcal{M}_g}
\newcommand{\Ccal}{\mathcal{C}}
\newcommand{\Acal}{\mathcal{A}}

\newcommand{\eps}{\epsilon}

\usepackage{overpic}

\title{Spectral gaps of double covers of hyperbolic surfaces}
\author{Lawford Hatcher}
\address[Lawford Hatcher]{D\'epartement de math\'ematiques et de statistique, Universit\'e de Montr\'eal, CP 6128 succ
Centre-Ville, Montr\'eal QC H3C 3J7, Canada}
\email{lawford.harrison.hatcher@umontreal.ca}
\author{Bram Petri}
\address[Bram Petri]{Institut de Math\'ematiques de Jussieu--Paris Rive Gauche and  Institut universitaire de France ; Sorbonne Universit\'e and Universit\'e Paris Cit\'e, CNRS, IMJ-PRG, F-75005 Paris, France}
\email{bram.petri@imj-prg.fr}
\date{\today}

\begin{document}

\begin{abstract}
We prove that, in most moduli spaces of orientable hyperbolic surfaces of finite area, there exist surfaces all of whose covers of degree two have a new small eigenvalue. On the other hand, we show that four punctured spheres always admit a double cover without new small eigenvalues and that there exist closed surfaces of any genus with the same property. We prove the last two of these results using a version of Cheeger's inequality adapted to double covers.
\end{abstract}

\maketitle

\begin{figure}[ht]
    \centering
    \begin{overpic}{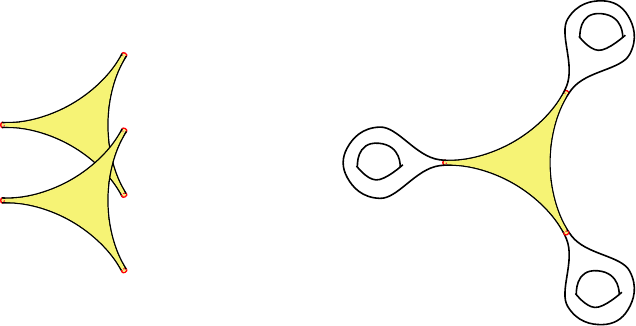}
    \put(32,24){{\Huge $\longrightarrow$}}
    \put(11,18.5){$\widehat{P}_1$}
    \put(10.5,30.5){$\widehat{P}_2$}
    \put(19,3.5){$\widehat{X}$}
    \put(99,-0.5){$X$}
    \put(81,25){$P$}
    \end{overpic}
    \caption{An embedded pair of pants $P\subset X$ that lifts to two disjoint pairs of pants $\widehat{P}_1,\widehat{P}_2 \subset \widehat{X}$ for every $2$-cover $\widehat{X}\to X$.}
    \label{pic_pants_lift}
\end{figure}

\section{Introduction}\label{intro}

In a series of breakthrough works \cite{MarcusSpielmanSrivastava_I,MarcusSpielmanSrivastava_IV}, Marcus, Spielman and Srivastava  proved that all finite graphs admit a cover of degree two whose adjacency matrix has no new eigenvalues above the spectral radius of their universal cover. This partially confirmed a conjecture by Bilu and Linial \cite{BiluLinial} and has since been generalized to all degrees \cite{HallPuderSawin}. Here, \textbf{new eigenvalues} are eigenvalues that do not come from lifts of the spectrum of the base graph. This, in particular, yielded the first proof of the existence of infinite sequences of bipartite Ramanujan graphs of every degree.

These results lead naturally to the following question: Does every hyperbolic surface of finite area admit a cover of degree two that has no new eigenvalues below $1/4$, the spectral gap of the hyperbolic plane? This question has been asked on multiple occasions (e.g. \cite{Sarnak_personal}, \cite[p. 12]{Hide_Thesis} and \cite{Magee_problem}) and its answer in the affirmative has been conjectured (e.g. \cite[Conjecture 1.6.2]{GolubevKamber}). If true, this would imply that there exist hyperbolic surfaces of arbitrarily large area with spectral gap at least $1/4$, which is currently not known.

Our first result answers this question in the negative. That is, the analog of Marcus, Spielman and Srivastava's theorem does not hold for hyperbolic surfaces:

\begin{thm}\label{badcoversthm}
    Let $(g,n)$ be such that either $g\geq 3$ and $n\geq 0$ or $g=2$ and $n\geq 1$. Then for any $\eps>0$, there exists a nonempty open set of surfaces $X\in \Mcal_{g,n}$ such that every double cover of $X$ has a new eigenvalue in $(0,\eps)$.
\end{thm}

In the theorem above, $\Mcal_{g,n}$ denotes the moduli space of hyperbolic surfaces of signature $(g,n)$, equipped with its topology induced by Fenchel--Nielsen coordinates. We build our surfaces by pinching the boundary geodesics of a subsurface $Y\subset X$ whose inclusion map induces the trivial map $H_1(Y,\Zbb) \to H_1(X,\Zbb)$, as illustrated in Figure \ref{pic_pants_lift}. In these examples, the base $X$ itself also has small eigenvalues. In particular, it may still hold that every hyperbolic surface of finite area admits a double cover with no new eigenvalues below $\min\{\lambda_1(X),1/4\}$. Here $\lambda_1(X)$ denotes the first non-zero element of the spectrum of $X$. Note that if $X$ is not closed, then $\lambda_1(X)$ is not necessarily an eigenvalue.

Our method does not work for surfaces that do not admit a homologically trivial subsurface of negative Euler characteristic. This is not just an artifact of the proof. Indeed, we prove that for four punctured spheres, the conjecture \emph{does} hold:

\begin{thm}\label{fourpuncturedthm}
    Every surface $X\in \Mcal_{0,4}$ has a double cover $X'\in \Mcal_{1,4}$ with no new eigenvalues in $[0,1/4)$.
\end{thm}

We note that there exist four punctured spheres with an arbitrarily small spectral gap---obtained by pinching a geodesic separating the four cusps into two pairs---as well as four punctured spheres with spectral gap $1/4$---that can be obtained, for instance, as covers of degree two of the thrice punctured sphere (see \cite[p. 9]{Magee_limit_points}). 

We prove Theorem \ref{fourpuncturedthm} using a version of Cheeger's inequality adapted to the new spectrum of a double cover (Theorem \ref{dccheegerineq} below). For every double cover of a complete Riemannian manifold $X$, we introduce an isoperimetric constant. This is defined in a similar way to the classical Cheeger constant, but the infimum is now taken over subsets that lift to two disjoint copies of themselves in the cover. The \textbf{double cover Cheeger constant} $h_{\mathrm{DC}}(X)$ is the supremum of these isoperimetric constants over all covers of degree two of $X$. We have
\[
\sup\{\lambda_1^{\mathrm{new}}(Y):\; Y\to X \text{ a cover of degree }2\} \geq \frac{1}{4} h_{\mathrm{DC}}(X)^2,
\]
where we have written $\lambda_1^{\mathrm{new}}(Y)$ for the bottom of the new spectrum of $Y$, which again is a slight abuse of notation, because it may not be an eigenvalue. For more details, including a proper definition, we refer the reader to Section \ref{cheegersection}.

Besides Theorem \ref{fourpuncturedthm}, we also use the double cover Cheeger constant to prove that there do exist families of closed hyperbolic surfaces with double covers all of whose new eigenvalues are above $\frac{1}{4}$:

\begin{thm}\label{goodcoversurfacesthm}
    For any $g\geq 2$, $\Mcalg$ contains a nonempty open set of surfaces $X$  with $h_{\mathrm{DC}}(X)>1$ and thus surfaces $X$ that admit a double cover without new eigenvalues in $[0,1/4)$. 
\end{thm}

The surfaces $X$ we produce have small eigenvalues themselves ($2g-2$ of them in fact), so the theorem above does not lead to infinite sequences of surfaces with a spectral gap of more than $\frac{1}{4}$. We also note the contrast with the classical Cheeger constant that cannot be used to produce closed surfaces of high genus with an optimal spectral gap (see \cite{BudzinskiCurienPetri}).

Finally, we observe that there do exist surfaces \emph{none} of whose non-trivial covers of degree two have new small eigenvalues, yet whose double cover Cheeger constant is below $1$. The Bolza surface, the most symmetric closed and orientable hyperbolic surface of genus $2$, has this property. It was proved in \cite[Section 3]{HidePetri} that none of the nontrivial degree-two covers of this surface have small eigenvalues, and we prove the following bound in Section \ref{bolzasection} below:
\begin{prop}\label{doublecheegerbolza}
The Bolza surface $X_{\mathrm{B}}$ satisfies
\[
h_{\mathrm{DC}}(X_{\mathrm{B}}) \leq 2\;\mathrm{arccosh}\left(1+\sqrt{2}\right)/\;\pi = 0.973118\ldots<1.
\]
\end{prop}

Together with Theorems \ref{badcoversthm} and \ref{dccheegerineq}, this proposition also implies that for every $g\geq 2$, $\Mcal_g$ contains a surface with double cover Cheeger constant below $1$.

We will finish this introduction with some further notes and references, but first we ask the reader to answer the questions that we did not answer. Indeed, Theorems \ref{badcoversthm} and \ref{fourpuncturedthm} above leave the question of whether every surface has a degree-two cover without small new eigenvalues open for infinitely many signatures:
\begin{quest}
    Does the statement of Theorem \ref{badcoversthm} extend to
    \begin{itemize}
        \item[(a)] surfaces of genus $0$ with at least $5$ cusps? 
        \item[(b)] surfaces of genus $1$?
        \item[(c)] closed surfaces of genus $2$?
    \end{itemize}
\end{quest}

Our theorems also fall short of answering the question of whether there exist hyperbolic surfaces of arbitrarily large area with spectral gap at least $1/4$. An affirmative answer to the following question would yield an affirmative answer to this question as well.

\begin{quest}
    Does the following hold for every finite-area hyperbolic surface $X$:
    \[\max\{\lambda_1^{\mathrm{new}}(Y):\; Y\to X \text{ a cover of degree }2\}\geq \min\{\lambda_1(X),1/4\}?\]
\end{quest}

\subsection{Further notes and references}

A lot of recent progress on the spectral geometry of hyperbolic surfaces was inspired by graph theory. One particularly fruitful source of inspiration is Friedman's theorem. Friedman proved that random high-degree covers of a wedge of at least two circles have a near-optimal spectral gap \cite{Friedman}, confirming a conjecture by Alon \cite{Alon}. This was later reproved and generalized to all finite base graphs in \cite{Bordenave, BordenaveCollins,ChenGarzaVargasTroppvanHandel}. Moreover, Huang, McKenzie and Yau \cite{HuangMcKenzieYau} recently proved that the renormalized distribution of the spectral gap converges to a Tracy-Widom distribution, which proves the existence of infinite sequences of Ramanujan graphs of any fixed degree.

The natural analog for hyperbolic surfaces of Friedman's theorem was proven by Hide and Magee \cite{HideMagee}: a random high-degree cover of a non-compact hyperbolic surface of finite area has a near-optimal relative spectral gap. More recently this was extended to covers of closed surfaces \cite{MageePudervanHandel,HideMaceraThomas1}. Some work also has been done to find necessary geometric conditions for Riemannian surface covers to lack new small eigenvalues \cite{BallmannMondal2025}. Besides random covers, analogous results are known for many models of random surfaces of large area \cite{AnantharamanMonk,HideMaceraThomas2,ShenWu}. The analog of Huang, McKenzie and Yau's theorem is still not known, but proving it for any of the models would lead to sequences of hyperbolic surfaces of arbitrarily large area without eigenvalues in $[0,\frac{1}{4}]$.

Finally, we note that it was recently proved that the full Bilu--Linial conjecture is false in general for graphs \cite{Xu}. Like in our case, the counterexample is a base graph with a very small spectral gap. Besides that, the counterexample is different because a connected subgraph is trivial in homology if and only if it is a subtree.

\subsection{Statement on the usage of artificial intelligence}

The mathematics for this paper was done in February and March 2026, when the machine (at least the version we had access to) was not yet a good enough mathematician to prove our theorems. As such, all the mathematical ideas are human. Nonetheless, we used ChatGPT for literature search and to check and test some computations.

\subsection{Acknowledgments}

Both of us thank  the Isaac Newton Institute for Mathematical Sciences, Cambridge, for support and hospitality during the programme \textit{Geometric spectral theory and applications}, where work on this paper was undertaken. This work was supported by EPSRC grant no EP/Z000580/1. LH thanks the Institut Universitaire de France for financing a trip to Paris and the Institut de Mathématiques de Jussieu--Paris Rive Gauche for its hospitality when some of the work for this paper was done. BP thanks the Simons Institute for the Theory of Computing, Berkeley, where, during the program \textit{Spectral Theory Beyond Graphs}, some of the writing for this paper was done.

\section{Proof of Theorem \ref{badcoversthm}}

The idea of the proof of Theorem \ref{badcoversthm} is to find a surface with a pair of pants with all boundary geodesics very short and which has a disconnected pre-image under every double cover. The result then follows from a simple test function argument. We work out the details in this section.

Given $\ell_1,\ell_2,\ell_3>0$, let $P(\ell_1,\ell_2,\ell_3)$ denote a hyperbolic pair of pants with boundary components $\gamma_1,\gamma_2,\gamma_3$ of length $\ell_1,\ell_2,\ell_3$, respectively. For $\ell>0$, let \[w(\ell)={\rm arcsinh}\bigg(\frac{1}{\sinh(\ell/2)}\bigg).\] For $i\in \{1,2,3\}$, let \[\Ccal_i:=\big\{x\in P\mid d(x,\gamma_i)<w(\ell_i)\big\},\] where $d$ denotes geodesic distance, denote the collar neighborhood of $\gamma_i$. By the collar theorem (see \cite{Keen} and \cite[Chapter 4]{Buser_book}), the sets $\Ccal_i$ are mutually disjoint topological cylinders. Each collar has area at most $2$ while $P$ has area equal to $2\pi$. Thus, the complement of the collars has area at least $2\pi-6$. \\
\indent If some of the $\ell_i$ equal $0$, then we define $P(\ell_1,\ell_2,\ell_3)$ similarly except with a cusp in place of a boundary curve for the corresponding $\ell_i$. The cusps also have a collar neighborhood, and the area estimate above generalizes to this case; see \cite[Section 4.4]{Buser_book}.

Given a double cover $p:S'\to S$, let $\lambda_1^{{\rm new}}(S')$ be the least element of the new spectrum of $S'$.

\begin{lem}\label{smallpantssmallevals}
    Let $\ell_1,\ell_2,\ell_3 \geq 0$, and suppose that $S$ is a hyperbolic surface that contains an isometric copy of the pair of pants $P=P(\ell_1,\ell_2,\ell_3)$. Let $p:S'\to S$ be a double cover of $S$ such that $p^{-1}(P)$ has two connected components, each isometric to $P$. Then 
    \[\lambda_1^{{\rm new}}(S')\leq \frac{2\pi}{(2\pi-6)\cdot \min\limits_{i=1,2,3:\;\ell_i\neq 0} \{w(\ell_i)^2\}},\]
    where we have used the convention that the minimum of the empty set equals $\infty$.
\end{lem}
\begin{proof}
    If all three boundary lengths are $0$, then $P$ is a thrice punctured sphere and thus forms a connected component of $S$. Our assumption then implies that $P$ lifts to two disconnected copies of itself in $S'$. So then $\lambda_1^{\rm new}(S')=0$, corresponding to the eigenfunction that equals $1$ on the first copy of $P$ in $S'$ and $-1$ on the other.

    Now we assume that at least one of the three boundary lengths is non-zero. By the variational characterization of the eigenvalue problem, we need only construct a function $u\in H^1(S')$ that is odd with respect to the non-trivial element of the deck group of the cover and with Rayleigh quotient less than the right-hand side. Let $P_1, P_2$ denote the connected components of $p^{-1}(P)$. Set $W=\min\limits_{i=1,2,3:\;\ell_i\neq 0} \{w(\ell_i)\}$. Then let \[u(x)=\begin{cases}
        \min\{d(x,\partial P_1),W\}\;\;&\text{if}\;\;x\in P_1\\
        -\min\{d(x,\partial P_2),W\}&\text{if}\;\;x\in P_2\\
        0&\text{otherwise.}
    \end{cases}\]
    Then $u\in H^1(S')$ with $|\nabla u| \in \{0,1\}$ almost everywhere in $P_1\cup P_2$ and with $|\nabla u|=0$ outside of $P_1\cup P_2$. It is also clear that $u$ is odd with respect to the non-trivial element of the deck group. By the above discussion, each $P_i$ contains a set of area at least $2\pi-6$ that is distance at least $W$ from $\partial P_i$. From these facts, we get 
    \[\int_{S'}|\nabla u|^2 \leq |P_1\cup P_2| = 4\pi\] and \[\int_{S'}|u|^2>2\cdot (2\pi-6)\cdot W^2.\] Dividing these computations yields the estimate.
\end{proof}

Note that as $\min_{\ell_i\neq 0} \{\ell_1,\ell_2,\ell_3\}$ approaches zero, the right-hand side of the estimate in Lemma \ref{smallpantssmallevals} also approaches zero. \\
\indent We now construct surfaces whose double covers all contain such a pair of pairs of pants. Let $P\subseteq S$ be a pair of pants, and let $i:P\hookrightarrow S$ denote the inclusion map. Given a double cover $p:S'\to S$ defined by a homomorphism $\phi:\pi_1(S)\to \Ztwo$, the inclusion $i$ lifts to a map $i':P\hookrightarrow S'$ if and only if \[i_*(\pi_1(P))\subseteq p_*(\pi_1(S'))=\ker(\phi).\] Note that $\phi$ factors through a map $\tilde{\phi}:H_1(S)\to \Ztwo$, where $H_1(S)$ is the first homology group of $S$ with coefficients in $\Ztwo$. Therefore, if \[i_*(H_1(P))=\{0\}\subseteq H_1(S),\] then $p^{-1}(P)$ has two connected components, each isometric to $P$, for every double cover. With this in hand, we can now prove Theorem \ref{badcoversthm}, which we first repeat for the convenience of the reader:
\begin{thmrep}{\ref*{badcoversthm}}
    Let $(g,n)$ be such that either $g\geq 3$ and $n\geq 0$ or $g=2$ and $n\geq 1$. Then for any $\eps>0$, there exists a nonempty open set of surfaces $X\in \Mcal_{g,n}$ such that every double cover of $X$ has a new eigenvalue in $(0,\eps)$.
\end{thmrep}

\begin{proof}
    Let $(g,n)$ be one of the signatures listed in the theorem and $\epsilon>0$. The main observation is that a surface of this signature contains a homologically trivial pair of pants, i.e. a pair of pants such that $i_*(H_1(P)) = \{0\}$. 
    
    Indeed, if $n= 0$ (and thus by assumption $g\geq 3$), choose $g_1,g_2,g_3>0$ such that $g_1+g_2+g_3=g$ and attach compact surfaces $S_1,S_2,S_3$ with one boundary component each and genera $g_1,g_2,g_3$, respectively, to the three boundary components of a compact pair of pants $P$. In the resulting surface $S$, $P$ is homologically trivial. If $n>1$, then we can do essentially the same thing, replacing $S_3$ by a surface of genus $0$, with one boundary component and $n$ cusps (and this works as soon as $g\geq 2$). If $n=1$, we start with a pair of pants $P$ with one cusp and two boundary components instead of a compact pair of pants.

    We now equip $S$ with any complete hyperbolic metric such that the induced metric on $P$ is isometric to $P(\ell_1,\ell_2,\ell_3)$, chosen such that \[\frac{2\pi}{(2\pi-6)\cdot \min\limits_{i=1,2,3:\;\ell_i \neq 0}\{w(\ell_i)^2\}}<\min\{1/4,\epsilon\}.\] 
    The set of possible choices of metrics induces an open subset of $\Mcal_{g,n}$.

    By the discussion above, for any double cover $p:S'\to S$, the pre-image $p^{-1}(P)$ has two connected components each isometric to $P$. The theorem then follows from Lemma \ref{smallpantssmallevals} and the fact that the spectrum of such a surface in $[0,1/4)$ consists of eigenvalues.
\end{proof}

\section{The double cover Cheeger constant}\label{cheegersection}

In this section we develop a version of Cheeger's inequality adapted to double covers. Because hyperbolic geometry and two-dimensionality do not play a role in this, we will briefly let go of both and let $X$ be a complete Riemannian manifold. 

%In what follows, every (co)homology group is taken to be over $\Zbb/2\Zbb$.\\
%\indent There is a correspondence between elements of ${\rm Hom}(H_1(X),\Ztwo)$ and the set of (Riemannian) double covers of $X$. Indeed, by \LH{cite hatcher}, there is a correspondence between the set of double covers of $X$ and index $2$ subgroups of $\pi_1(X)$, where the subgroup corresponding to a given cover is the image of the pushforward of its fundamental group by the covering map. Since index $2$ subgroups are normal, these subgroups are in bijection with homomorphisms $\pi_1(X)\to \Ztwo$. Since $\Ztwo$ is abelian, such homomorphisms factor through the abelianization and$\mod 2$ reduction $\pi_1(X)\to H_1(X)$, establishing the correspondence.\\
%\indent In the special case that $X$ is a compact Riemannian surface, the algebraic intersection form $i:H_1(X)\times H_1(X)\to \Ztwo$ defines a further bijection between $H_1(X)$ and ${\rm Hom}(H_1(X),\Ztwo)$ given by $a\mapsto i(a,\cdot )$. Given $a \in H_1(X)$, the double cover defined by this correspondence can be realized geometrically as follows. Let $\alpha$ be a simple closed curve representing $a$. Let $X_1$ and $X_2$ be two isometric copies of $X$. Cut each $X_i$ along its copy of $\alpha$, and glue $X_1$ to $X_2$ along these cuts in the orientation-preserving way to obtain a closed Riemannian surface. The resulting surface is the double cover. 

Let $H_1(X)$ denote the first homology group of $X$ with coefficients in $\Ztwo$. Given an element $\varphi\in {\rm Hom}(H_1(X),\Ztwo)$, denote the corresponding double cover by $p=p_{\varphi}:X_{\varphi}\to X$. Let $\sigma:X_{\varphi}\to X_{\varphi}$ denote the nontrivial element of the deck group of the cover, and let \[L^2_{\varphi}(X_{\varphi}):=\{u\in L^2(X_{\varphi})\mid \sigma^*u=-u\}.\] 
The restriction of the Laplacian on $X_{\varphi}$ to $H^2(X_{\varphi})\cap L^2_{\varphi}(X_{\varphi})$ defines a self-adjoint operator, and we call its spectrum the \textbf{new spectrum} of $X_{\varphi}$. We call the eigenvalues in the new spectrum the \textbf{new eigenvalues} of $X_{\varphi}$. In the case that $X$ (resp. $X_{\varphi}$) is not compact, the set of new eigenvalues may be empty. We denote the infimum of the new spectrum by $\lambda_1^{\varphi}$. Note that, by the min-max theorem, \begin{equation}\label{minmaxnew}\lambda_1^{\varphi}=\inf_{u\in L^2_{\varphi}(X_{\varphi})\cap C^{\infty}_0(X_{\varphi})\setminus \{0\}}\frac{\int_{X_{\varphi}}|\nabla u|^2}{\int_{X_{\varphi}}|u|^2}.\end{equation}\\
\indent We now show that there exists a Cheeger-type inequality for $\lambda_1^{\varphi}$. We say that a nonempty open subset $A\subseteq X$ is \textbf{$\varphi$-admissible} if $\partial A$ is a smooth hypersurface and either of the following equivalent conditions holds: 
\begin{enumerate}
    \item For each connected component $A'\subseteq A$, the pre-image $p^{-1}(A')$ is disconnected.
    \item If $\iota_A:A\hookrightarrow X$ denotes the inclusion map, then the image $(\iota_A)_*(H_1(A))\subseteq H_1(X)$ lies in the kernel of $\varphi$. 
\end{enumerate}
\indent For instance, if $A$ has trivial first homology, then $A$ is $\varphi$-admissible for any $\varphi$. Moreover, a smoothly bounded, nonempty open subset of a $\varphi$-admissible set is $\varphi$-admissible. For a less trivial example, suppose that $X$ is a closed Riemannian surface and that $\varphi$ is defined to be $i(a,\cdot)$ for some $a\in H_1(X)$, where $i(\cdot,\cdot)$ denotes the (mod $2$) intersection form on $H_1(X)$. Then a tubular neighborhood of a simple closed curve $\alpha$ representing $a$ is $\varphi$-admissible, and the set $X\setminus \alpha$ is also $\varphi$-admissible.   \\
\indent Given a smoothly bounded open subset $A\subseteq X$, let $|\Bpartial A|$ denote the hypersurface area of the boundary of the abstract metric completion of $A$. By the last example in the preceding paragraph, it is not always the case that $|\Bpartial A|$ equals $|\partial A|$. The \textbf{$\varphi$-Cheeger constant} of $X$ is then 
\[h_{\varphi}(X):=\inf\bigg\{\frac{|\Bpartial A|}{|A|}\mid A\subseteq X\;\;\varphi\text{-admissible}\bigg\}.\] 
The \textbf{double cover Cheeger constant} of $X$ is defined as \[h_{\mathrm{DC}}(X):=\sup_{\varphi\in{\rm Hom}(H_1(X),\Ztwo)}h_{\varphi}(X).\] 
\begin{remk}\label{existenceremark}
    Using the discreteness of the length spectrum of a hyperbolic surface, one can show that below any fixed threshold $L>0$, there are at most finitely many homotopy classes of $\varphi$-admissible subsets with isoperimetric ratio at most $L$. Applying a standard Arzel\`a-Ascoli argument to the boundary curves of a minimizing sequence of $\varphi$-admissible subsets then implies the existence of a minimizing subset for the $\varphi$-Cheeger constant. In particular, the double cover Cheeger constant of a hyperbolic surface is realized by some subset.
\end{remk}

The double cover Cheeger inequality now reads:

\begin{thm}\label{dccheegerineq}
    For any complete Riemannian manifold with a double cover defined by $\varphi\in {\rm Hom}(H_1(X),\Ztwo)$, \[\lambda_1^{\varphi}\geq \frac14\cdot h_{\varphi}(X)^2.\] In particular, \[\sup\{\lambda_1^{\mathrm{new}}(Y):\; Y\to X \text{ a cover of degree }2\}\geq \frac14\cdot h_{DC}(X)^2.\]
\end{thm}

Every double cover is regular with deck group isomorphic to $\Ztwo$. Let $u_{\varphi}\in L^2_{\varphi}(X_{\varphi})\cap C_0^{\infty}(X_{\varphi})$. Then for each $x\in X$ with fiber $p^{-1}(x)=\{x_1,x_2\}$, we have $u_{\varphi}(x_1)=-u_{\varphi}(x_2)$. Thus, we can define a function $u$ on the base $X$ by setting $u(x)=|u_{\varphi}(x_1)|$ with $x_1\in p^{-1}(x)$ chosen arbitrarily. Since $u_{\varphi}$ is smooth, $u$ lies in the Sobolev space $H^1(X)$ and is smooth away from its zero-level set, and we have
\begin{equation}\label{rayleighequal}
    \frac{\int_{X}|\nabla u|^2}{\int_{X}|u|^2}=\frac{\int_{X_{\varphi}}|\nabla u_{\varphi}|^2}{\int_{X_{\varphi}}|u_{\varphi}|^2}.
\end{equation}

\begin{lem}\label{superadmissible}
    The nonempty super-level sets $A_t=\{u(x)> t\}$ are $\varphi$-admissible for almost every $t\geq 0$. 
\end{lem}
\begin{proof}
    By Sard's theorem, almost every $A_t$ has smooth boundary. Let $A_t'\subseteq A_t$ be a nonempty connected component. Since $u_{\varphi}$ is odd with respect to the non-trivial deck transformation, $u_{\varphi}$ takes on both positive and negative values on $p^{-1}(A_t')$. Since $u_{\varphi}$ is continuous and does not vanish on $p^{-1}(A_t')$, the set $p^{-1}(A_t')$ is disconnected. 
\end{proof}

Given Lemma \ref{superadmissible}, the proof of Theorem \ref{dccheegerineq} follows from essentially the same computation used to prove the traditional Cheeger inequality. 

\begin{proof}[Proof of Theorem \ref{dccheegerineq}]
    Let $u_{\varphi}$ and $u$ be as above. Note that (by Sard's theorem) for almost every super-level set $A_t$ of $u$, we have $|\Bpartial A_t|=|\partial A_t|$. By Lemma \ref{superadmissible}, we therefore have \[|\partial A_t|\geq h_{\varphi}(X)|A_t|.\]
    Since the function $u_{\varphi}$ is smooth and compactly supported, the function $u^2$ is Lipschitz, so we may apply the coarea formula:
    \begin{align*}
        \int_X|\nabla (u^2)|=\int_{0}^{\infty}|\partial A_{\sqrt{t}}|dt\geq h_{\varphi}(X)\int_0^{\infty}|A_{\sqrt{t}}|=h_{\varphi}(X)\int_{X}|u|^2. 
    \end{align*}
    By (\ref{rayleighequal}), we also have \[\int_{X}|\nabla (u^2)|=2\int_{X}|u|\cdot |\nabla u|\leq 2\|u\|_{L^2(X)}\|\nabla u\|_{L^2(X)}.\] Combining these two inequalities gives \[\frac{\int_{X}|\nabla u|^2}{\int_{X}|u|^2}\geq \frac14 h_{\varphi}(X)^2.\] Since this inequality holds for every odd smooth function with compact support on $X_{\varphi}$, Equation (\ref{rayleighequal}) along with the variational characterization (\ref{minmaxnew}) of $\lambda_1^{\varphi}$ implies the result. 
\end{proof}

\section{Hyperbolic surfaces with double covers with no new small eigenvalues}

In this section, $X$ will always be a complete, finite-area hyperbolic surface. We will prove Theorems \ref{fourpuncturedthm} and \ref{goodcoversurfacesthm}. The following definition and lemma will allow us to more easily prove lower bounds on double cover Cheeger constants in some situations.

\begin{defn}
    A $\varphi$-admissible set $A\subseteq X$ of a hyperbolic surface is \textbf{reduced} if it is connected and one of the following holds:
    \begin{enumerate}
        \item $A$ is a topological disk, or 
        \item No curve in $\partial A$ bounds a topological disk in $X$. 
    \end{enumerate}
\end{defn}

\begin{lem}\label{reducedlemma}
    Let $A\subseteq X$ be $\varphi$-admissible. If $A$ is not reduced, then there exists a reduced $\varphi$-admissible set $A'\subseteq X$ for which \[\frac{|\overline\partial A'|}{|A'|}\leq \frac{|\overline\partial A|}{|A|}.\]
\end{lem}

\begin{proof}
    Recall that for any four positive numbers $a,b,c,d>0$, we have \[\min\Big\{\frac ab,\frac cd\Big\}\leq \frac{a+c}{b+d}.\] This inequality immediately implies that if $A_0\subseteq A$ is the connected component of $A$ having minimal isoperimetric ratio, then its isoperimetric ratio is at most that of $A$. Let $A'$ equal the union of $A_0$ with the closed topological disks bounded by curves in $\partial A_0$. Then clearly $A'$ has smaller perimeter and larger area than $A_0$, so we have 
    \[\frac{|\overline\partial A'|}{|A'|}\leq \frac{|\overline\partial A_0|}{|A_0|}\leq \frac{|\overline\partial A|}{|A|}.\] It therefore remains to show that $A'$ is $\varphi$-admissible (because it is then also reduced). Since $\iota_{A_0}:A_0\hookrightarrow X$ factors through the map $\iota_{A'}:A'\hookrightarrow X$, we have a factorization \[H_1(A_0)\longrightarrow H_1(A')\xrightarrow{(\iota_{A'})_*} H_1(X).\] To show that the image of $H_1(A')$ lies in the kernel of $\varphi$, it suffices to show that the first map $H_1(A_0)\to H_1(A')$ is surjective since $A_0$ is $\varphi$-admissible. This follows quickly from a Mayer-Vietoris argument. 
  %  \LH{Recording the Mayer-Vietoris argument in case it isn't sufficiently obvious:
  %  Let $B$ be an open neighborhood of the disks we added to $A_0$ such that the intersection $A_0\cup B$ deformation retracts to the boundary components of $A_0$ to which we added the disks. Further suppose that $B$ is also a union of topological disks. The Mayer-Vietoris sequence for $A'=A_0\cup B$ gives \[\dots\rightarrow H_1(A_0)\oplus H_1(B)\cong H_1(A_0)\rightarrow H_1(A')\rightarrow H_0(A_0\cap B)\rightarrow H_0(A_0)\oplus H_0(B)\rightarrow\dots \]  The map \[H_0(A_0\cap B)\longrightarrow H_0(A_0)\oplus H_0(B)\] induced by inclusion is clearly injective, so the map $H_1(A')\to H_0(A_0\cap B)$ is the zero map. This implies that the map $H_1(A_0)\to H_1(A')$ is surjective, but this map is exactly the map induced by the inclusion $\iota:A_0\hookrightarrow A'$.}
\end{proof}

By Lemma \ref{reducedlemma}, we could redefine the $\varphi$-Cheeger constant using only reduced admissible sets. \\
\indent The next lemma is well known (see for instance \cite[Theorem 2.2]{AdamsMorgan}) and shows that, if we wish to prove that some $\varphi$-Cheeger constant is greater than $1$, then we need only consider sets that are not topological disks or annuli.

\begin{lem}\label{diskandannulus}
    Let $A\subseteq X$ be a topological disk or annulus not homotopic to a cusp. Then \[\frac{|\overline\partial A|}{|A|}>1.\]
\end{lem}
%\begin{proof}
%    If $A$ is a topological disk, then it is isometric to a subset of the hyperbolic plane, and this follows from the hyperbolic isoperimetric inequality. If $A$ is an annulus, then let $\tilde A\subseteq \Hbb^2$ be a connected component of a preimage of $A$ under some universal covering map. If $A$ is contractible in $X$, then $\tilde A$ is an annulus, and we are once again done by the hyperbolic isoperimetric inequality. Otherwise, \LH{the following proof was given to me by chatgpt. Is there a reference for this?} let $\gamma\in \pi_1(X)$ be a class corresponding to a boundary curve of $A$. Let $p:X'\to X$ be the infinite cyclic cover corresponding to the subgroup $\langle \gamma\rangle \leq \pi_1(X)$. Then $X'$ is a hyperbolic cylinder containing a lift $A'$ of $A$ with area and perimeter equal to those of $A$. On $X'$, we may compute in Fermi coordinates $(r,t)\in \Rbb\times (\Rbb/L\Zbb)$ for some $L>0$ (equal to the length of the closed geodesic homotopic to the boundary curves of $A$) with the metric \[ds^2=dr^2+\cosh^2(r)dt^2.\] Define a vector field $V=\tanh(r)\partial_r$ on $X'$. Then $|V|=|\tanh(r)|<1$ for all $(r,t)$, and the divergence of $V$ is identically equal to $1$. By the divergence theorem, we get \[|A|=|A'|=\int_{A'}1=\int_{A'}{\rm div}(V)=\int_{\partial A'}\langle V,\nu\rangle <\int_{\partial A'}1=|\partial A'|=|\partial A|.\]
%\end{proof}

We are now ready to prove that all four punctured spheres have good double covers:
\begin{thmrep}{\ref*{fourpuncturedthm}}
        Every surface $X\in \Mcal_{0,4}$ has a double cover $X'\in \Mcal_{1,4}$ with no new eigenvalues in $[0,1/4)$.
\end{thmrep}
\begin{proof}
    Let $X$ be a hyperbolic four punctured sphere. By Theorem \ref{dccheegerineq}, it suffices to prove that $h_{\mathrm{DC}}(X)>1$. Let $\alpha_1,\alpha_2,\alpha_3$ be disjoint simple closed curves, each homotopic to a distinct cusp. For $i=1,2,3$, let $a_i\in H_1(X)$ denote the homology class of $\alpha_i$. Then $H_1(X)\cong (\Ztwo)^3$ is generated by the $a_i$. \\
    \indent Let $\varphi:H_1(X)\to \Ztwo$ be the map defined by sending each $a_i$ to $1$. Under the double cover $X_{\varphi}$ corresponding to this map, a neighborhood of any cusp is \textit{not} $\varphi$-admissible, so $X_{\varphi}$ has four cusps and, therefore, genus one. \\
    \indent By Lemmas \ref{reducedlemma} and \ref{diskandannulus} it suffices to show that any reduced $\varphi$-admissible set $A\subseteq X$ is a topological disk or an annulus homotopic to a cusp. Let $A\subseteq X$ be a reduced $\varphi$-admissible set that is not a topological disk. Since $A$ is $\varphi$-admissible, it does not have a cusp, and none of its boundary components are homotopic to a cusp. In particular, $\partial A$ has at least two connected components $\gamma_1,\gamma_2$. Each $\gamma_j$ is a simple closed curve whose homology class lies in the kernel of $\varphi$. Since $X$ is a four punctured sphere, if neither $\gamma_j$ is nullhomotopic, they must be homotopic to each other (else they intersect). Since $A$ is reduced, it must be a topological annulus. 
\end{proof}

We can now also prove Theorem \ref{goodcoversurfacesthm}:

\begin{thmrep}{\ref*{goodcoversurfacesthm}}
    For any $g\geq 2$, $\Mcalg$ contains a nonempty open set of surfaces $X$  with $h_{\mathrm{DC}}(X)>1$ and thus surfaces $X$ that admit a double cover without new eigenvalues in $[0,1/4)$. 
\end{thmrep}

\begin{proof}
    First we choose a particular pants decomposition on a closed surface of genus $g$. The easiest way to describe it is with a picture, which can be found in Figure \ref{pic_thinsurf}. For the reader who prefers to be confused by a formal description: the pants decomposition is made up of simple closed curves $\{\alpha_i\}_{i=1}^g\cup \{\alpha_j'\}_{j=2}^{g-1}\cup \{\gamma_k\}_{k=1}^{g-1}$, such that:
    \begin{itemize}
    \item $\alpha_1$ and $\gamma_1$ bound a pair of pants $P_1$ and $\alpha_1$ is incident to $P_1$ on both sides,
    \item for all $k=2,\ldots, g-1$, $\gamma_{k-1}$, $\alpha_k'$ and $\alpha_k$ bound a pair of pants $P_{2k-2}$,
    \item for all $k=2,\ldots, g-1$,  $\alpha_k'$ and $\gamma_k$ and $\alpha_k$ bound a pair of pants $P_{2k-1}$, and
    \item $\gamma_{g-1}$ and $\alpha_g$ bound a pair of pants $P_{2g-2}$ in such a way that $\alpha_g$ is incident to $P_{2g-2}$ on both sides.
    \end{itemize}
    
    \begin{figure}
        \vspace*{0.5cm}
        \centering
        \begin{overpic}{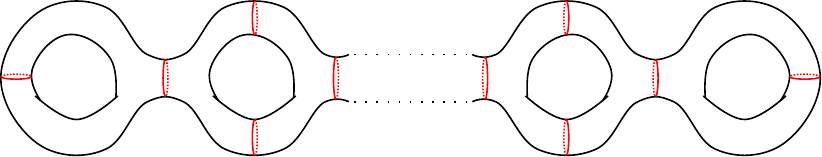}
        \put(-3.5,9){$\alpha_1$}
        \put(19,4){$\gamma_1$}
        \put(29.5,-2.5){$\alpha_2$}
        \put(29.5,20.5){$\alpha_2'$}
        \put(40,4){$\gamma_2$}
        \put(56,4){$\gamma_{g-2}$}
        \put(67.5,-2.5){$\alpha_{g-1}$}
        \put(67.5,20.5){$\alpha_{g-1}'$}
        \put(77.5,4){$\gamma_{g-1}$}
        \put(100.5,9){$\alpha_g$}
        \end{overpic}
        \caption{A surface with a linear pants decomposition}
        \label{pic_thinsurf}
    \end{figure}

    Now set
    \[
    L(g) = 2\sinh^{-1}\Big(\frac{1}{\sinh(2\pi(g-1))}\Big).
    \] 
    Let $\Acal_g \subset \Mcal_g$ denote the set of surfaces that admit a pants decomposition of the topological type described above in which all curves have length below $L(g)$, with arbitrary twists. This is an open subset of $\Mcal_g$. 
    
    Now let $X\in\Acal_g$ be any surface, equipped with its short pants decomposition. The set of boundary curves of the $P_i$ in $X$ contains a subset of disjoint simple closed geodesics $\{\alpha_1,\dots,\alpha_g\}$ whose images $\{a_1,\ldots,a_g\}$ in $H_1(X)$ generate a Lagrangian subspace for the intersection form. Moreover, the inclusion map for each pair of pants induces a nontrivial map $H_1(P_i)\to H_1(X)$.
    
    Let $\varphi:H_1(X)\to \Ztwo$ be a homomorphism mapping each $a_i$ to $1$. Then none of the images of the maps $H_1(P_i)\to H_1(X)$ are contained in $\ker(\varphi)$. By Lemma \ref{reducedlemma} we need only show that any reduced $\varphi$-admissible subset has isoperimetric ratio greater than $1$. Suppose that $A\subseteq X$ is $\varphi$-admissible and reduced with $|\overline\partial A|\leq |A|$. By Lemma \ref{diskandannulus}, $A$ is neither a topological disk nor a topological annulus. In particular, no boundary curve of $A$ is nullhomotopic. Moreover, we have $|\overline\partial A|\leq |X|=4\pi(g-1)$. By the collar lemma and the assumption on the boundary lengths of the $P_i$, each curve in $\partial A$ is homotopic to one of the curves defining the pants decomposition of $X$. Therefore, since $A$ is not an annulus, it admits a (not necessarily geodesic) pants decomposition consisting of curves homotopic to the curves used to define the pants decomposition of $X$. But then the map $H_1(A)\to H_1(X)$ induced by inclusion is not in the kernel of $\varphi$, contradicting that $A$ is $\varphi$-admissible. 
\end{proof}

\section{The double cover Cheeger constant of the Bolza surface}\label{bolzasection}

In this final section, we prove the bound on the double cover Cheeger constant on the Bolza surface that we stated in the introduction:
\begin{prprep}{\ref*{doublecheegerbolza}}
The Bolza surface $X_{\mathrm{B}}$ satisfies
\[
h_{\mathrm{DC}}(X_{\mathrm{B}}) \leq 2\;\mathrm{arccosh}\left(1+\sqrt{2}\right)/\;\pi = 0.973118\ldots.
\]
\end{prprep}

\begin{proof}
    We will first prove that every non-trivial element of $H_1(X_{\mathrm{B}})$ (with coefficients in $\Ztwo$) can be represented by a simple closed multi-geodesic whose length is at most $4 \;\mathrm{arccosh}\left(1+\sqrt{2}\right)$.
    
    To do so, recall that the Bolza surface can be obtained by gluing together the opposite sides of a regular octagon with interior angles $\pi/4$. There are multiple ways to use this picture to extract a basis for $H_1(X_{\mathrm{B}})$. We will use the side pairings. The four M\"obius transformations $g_1,g_2,g_3,g_4$ that pair the sides of the octagon in the correct order (for any fixed embedding of the octagon in the hyperbolic plane) generate a Fuchsian group $\Gamma_{\mathrm{B}}$ that uniformizes $X_{\mathrm{B}}$. The images $a_1,a_2,a_3,a_4\in H_1(X_{\mathrm{B}})$ of these generators under the composition $\Gamma_{\mathrm{B}}\to H_1(X_{\mathrm{B}})$ of the abelianization map with reduction modulo $2$ form a (non-symplectic) basis of $H_1(X_{\mathrm{B}})$.

    Now we need to know something about the geometry of the Bolza surface (or the octagon). We will label the sides of the octagon in order by $s_1,s_2,s_3,s_4,s_1',s_2',s_3',s_4'$.    We refer to \cite{Jenni,Schmutz} for proofs of the following geometric claims:
    \begin{enumerate}
        \item The side pairings themselves (corresponding to $a_i$ for $i=1,2,3,4$) generate closed loops that are homotopic to the projections of the orthogeodesics connecting opposite sides $s_i$ and $s_i'$. These are systoles and have length $2\; \mathrm{arccosh}\left(1+\sqrt{2}\right)$. 
        \item The loops formed by the pair of segments connecting $s_i$ to $s_{i+1}$ and $s_{i+1}'$ to $s_i'$ for $i=1,2,3,4$ (indices are considered modulo $4$ here) are homotopic to systoles as well. Because of the duality of the sides to the homology classes, these loops represent the classes $a_i+a_{i+1}$, for $i=1,2,3,4$.
        \item The sides of the octagon themselves project to systoles as well. By pushing $s_i$ off of the vertex, one obtains a loop that runs through all the pairs of sides, except $s_i$ itself. As such, all classes of the form $a_i+a_j+a_k$ with $i$, $j$ and $k$ distinct can be represented by systoles.
        \item The remaining classes can all be obtained as sums of two of the classes above and, as such, they end up on the right-hand side of the table below.
    \end{enumerate}
    We note that at least $a_1+a_2+a_3+a_4$ can be represented by a loop that is slightly shorter than twice the systole (see \cite[Section 7.4]{FortierBourquePetri}). But twice the systole is enough for what follows, so we did not make the effort to try to represent the remaining loops. 
    
    All in all, we obtain the following table of upper bounds on the lengths of simple closed multi-geodesics representing our homology classes:
    \begin{center}
        \begin{tabular}{l|l}
        \textbf{One systole} & \textbf{Two systoles} \\
        \hline
         $a_1$,  $a_2$, $a_3$, $a_4$,   &  $a_1+a_3$, $a_2+a_4$ \\
         $a_1+a_2$, $a_2+a_3$, &  $a_1+a_2+a_3+a_4$\\
         $a_3+a_4$, $a_4+a_1$    &  \\
         $a_1+a_2+a_3$, $a_1+a_2+a_4$, & \\
         $a_1+a_3+a_4$, $a_2+a_3+a_4$ 
        \end{tabular}
    \end{center}

    Now we use the fact that for every non-trivial homomorphism $\varphi:H_1(X_{\mathrm{B}})\to\Ztwo$, there exists a homology class $a\in H_1(X_{\mathrm{B}})$ such that $\varphi = i(a,\cdot)$. In particular, that means that if $\alpha$ is a simple closed multi-geodesic that represents $a$, then $X_{\mathrm{B}}\setminus \alpha$ is $\varphi$-admissible and $|\overline\partial (X_{\mathrm{B}}\setminus \alpha)|=2\cdot \ell(\alpha)$. As such
    \[
    h_\varphi(X_{\mathrm{B}}) \leq \frac{2\cdot \ell(\alpha)}{\mathrm{area}(X_{\mathrm{B}})} \leq \frac{8\;\mathrm{arccosh}\left(1+\sqrt{2}\right)}{4\pi},
    \]
    as claimed.
    
\end{proof}

\bibliographystyle{alpha}
\bibliography{refs}

\end{document}